\documentclass[11pt,a4paper]{amsart}
\usepackage{color,amsthm,amsfonts,amsmath,hyperref}
\usepackage[numbers,sort&compress]{natbib}
\usepackage[lmargin=40mm,rmargin=40mm,tmargin=40mm,bmargin=40mm]{geometry}

\title{The Big-Line-Big-Clique Conjecture\\ is False for Infinite Point  Sets}


\author{Attila~P\'or}
\address{\begin{minipage}{8cm}
Department of Mathematics\newline 
Western Kentucky University\newline
Bowling Green,  Kentucky, U.S.A.\\[-1ex]
\end{minipage}}
\email{attila.por@wku.edu}
\author{David~R.~Wood}
\address{\begin{minipage}{8cm}
Department of Mathematics and Statistics\newline  
The University of Melbourne\newline 
Melbourne, Australia\\[-1ex]
\end{minipage}}
\email{woodd@unimelb.edu.au}
\thanks{David Wood is supported by a QEII Research Fellowship from the Australian Research Council.}

\theoremstyle{plain}
\newtheorem{theorem}{Theorem}

\begin{document}
\maketitle

Let $P$ be a finite set of points in the plane. Two distinct points
$v$ and $w$ in the plane are \emph{visible} with respect to $P$ if no
point in $P$ is in the open line segment
$\overline{vw}$. \citet{KPW-DCG05} made the following Ramsey-theoretic
conjecture, which has recently received considerable attention
\citep{KPW-DCG05,Luigi,EmptyPentagon-GC,ABBCLPW,PorWood-JoCG,Matousek09}. 

\medskip\noindent{\bf Big-Line-Big-Clique Conjecture}
\citep{KPW-DCG05} For all $k\geq2$ and $\ell\geq2$ there is
an integer $n$ such that every finite set of at least $n$
points in the plane:
\begin{itemize}
\item contains $\ell$ collinear points, or
\item contains $k$ pairwise visible points.
\end{itemize}

\medskip
This conjecture is true for $k\leq 5$ or $\ell\leq3$
\citep{KPW-DCG05,Luigi,EmptyPentagon-GC}, and is open for $k=6$ or
$\ell=4$. Note that the natural approach for attacking this conjecture
using extremal graph theory fails \citep{PorWood-JoCG}. Another natural
approach for attacking the Big-Line-Big-Clique Conjecture is to follow
an infinitary compactness argument (which can be used to establish
many results in Ramsey theory). The purpose of this note is to show
that this conjecture is false for infinite point sets, which
suggests that an infinitary compactness argument cannot work.

\begin{theorem}
  There is a countably infinite point set with no 4 collinear points
  and no 3 pairwise visible points.
\end{theorem}

\begin{proof}
Let $x_1,x_2,x_3$ be three non-collinear points in the plane. 
  Given points $x_1,\dots,x_{n-1}$, define $x_n$ as follows.  By the Sylvester-Gallai
  theorem, there is a line through exactly two of
  $x_1,\dots,x_{n-1}$. Choose such a line
  $\overleftrightarrow{x_ix_j}$ with $i<j$ such that $j$ is minimum
  and then $i$ is minimum. Insert $x_n$ on 
  $\overline{x_ix_j}$, such that $\{x_i,x_n,x_j\}$ is the only
  collinear triple that contains  $x_n$. This is possible, since there are only
  finitely  many ($\leq\binom{n-3}{2}$) excluded locations for $x_n$.  

Repeat this  step to obtain a point set $\{x_i:i\in\mathbb{N}\}$,
which by construction,  contains no 4 collinear points. Moreover, if
$x_i$ and $x_k$ are visible with
  $i<k$, then $x_i$ and $x_k$ are collinear with some other point
  $x_{i'}$ (otherwise some point would be added at a later
  stage in $\overline{x_ix_k}$). Since $i<k$ we have 
$x_k\in\overline{x_ix_{i'}}$ and $i'<k$.  

Suppose on the contrary that three points $x_i,x_j,x_k$ are
  pairwise visible, where $i<j<k$. As proved above,
  $x_k\in\overline{x_ix_{i'}}$ and $x_k\in\overline{x_jx_{j'}}$, where
  $i',j'<k$. Since $x_k$ is in only one collinear
  triple amongst $x_1,\dots,x_k$, we have $i'=j$ and $j'=i$. Thus $x_i,x_k,x_j$ are
  collinear, and $x_i$ and $x_j$ are not visible. This contradiction
  proves that no 3 points are pairwise visible.
\end{proof}


\def\cprime{$'$} \def\soft#1{\leavevmode\setbox0=\hbox{h}\dimen7=\ht0\advance
  \dimen7 by-1ex\relax\if t#1\relax\rlap{\raise.6\dimen7
  \hbox{\kern.3ex\char'47}}#1\relax\else\if T#1\relax
  \rlap{\raise.5\dimen7\hbox{\kern1.3ex\char'47}}#1\relax \else\if
  d#1\relax\rlap{\raise.5\dimen7\hbox{\kern.9ex \char'47}}#1\relax\else\if
  D#1\relax\rlap{\raise.5\dimen7 \hbox{\kern1.4ex\char'47}}#1\relax\else\if
  l#1\relax \rlap{\raise.5\dimen7\hbox{\kern.4ex\char'47}}#1\relax \else\if
  L#1\relax\rlap{\raise.5\dimen7\hbox{\kern.7ex
  \char'47}}#1\relax\else\message{accent \string\soft \space #1 not
  defined!}#1\relax\fi\fi\fi\fi\fi\fi} \def\Dbar{\leavevmode\lower.6ex\hbox to
  0pt{\hskip-.23ex\accent"16\hss}D}

\bigskip

\end{document}